\documentclass[12pt]{article}
\makeatletter
        \def\@thefnmark{\null}
        \def\footnotetexta{\@footnotetext}

\title{Constructive methods for spectra with three nonzero elements in the nonnegative inverse eigenvalue problem}

\author{Anthony G. Cronin\footnote{School of Mathematics and Statistics, University College Dublin, Ireland,  anthony.cronin@ucd.ie}}

\date{}
\usepackage{amsmath,amsfonts,mathrsfs,amssymb,amsthm}
\newtheorem{Theo}{Theorem}

\newtheorem{defn}[Theo]{Definition}

\newtheorem{Pro}[Theo]{Proposition}
\newtheorem{Rem*}[Theo]{Remark}

\newcommand{\npmatrix}[1]{\begin{bmatrix} #1 \end{bmatrix}}

\usepackage{enumerate}

\usepackage{color,epsfig,graphicx}
\usepackage{setspace}
\def\bbbc{\mathpalette{}{{\setbox0=\hbox{$\rm C$}\hbox{\hbox
to0pt{\kern0.4\wd0\vrule height0.9\ht0\hss}\box0}}}}

\begin{document}

\maketitle

\begin{abstract}
\noindent
We present and compare three constructive methods for realizing non-real spectra with three nonzero elements in the nonnegative inverse eigenvalue problem. \\
We also provide some necessary conditions for realizability and numerical examples. In particular, we utilise the companion matrix.\\
\\
\emph{AMS Subject Classification:} 15A18; 15A29; 47A75; 58C40 \\
\emph{Keywords:} Nonnegative matrices; Inverse eigenvalue problem; Spectral theory
\end{abstract}

\section{Introduction}
We say that a list of $n$ complex numbers $\sigma =(\lambda_1, \lambda_2, \ldots, \lambda_n)$ is \emph{realizable} if $\sigma$ is the list of eigenvalues (including multiplicities) of an $n \times n$ entry-wise nonnegative matrix $A$ and that $\sigma$ \emph{realizes} $A$ or that $A$ is a realizing matrix for $\sigma$.\\
Let $$s_k:=\lambda_1^k+\lambda_2^k+\cdots+\lambda_n^k, \ k = 1, 2, 3, \ldots$$
Then 
\begin{align}
s_k=tr(A^k)\geq 0, \ k = 1, 2, 3, \ldots
\end{align}
is a necessary condition for $\sigma$ to be realizable.\\
The JLL conditions discovered independently by Johnson \cite{Jo} and Loewy and London \cite{LL} state that
\begin{align} \label{JLL}
n^{m-1}s_{km}\geq s_k^m, \ k,m = 1, 2, 3, \ldots
\end{align}
A further necessary condition for realizability of $\sigma$ comes from the Perron-Frobenius theorem \cite{Ber}, that is that there
exists $j$ with $\lambda_j$ real and $\lambda_j\geq |\lambda_i |$, for all $i$. Such a $\lambda_j$ is called the \emph{Perron root} of $\sigma$, usually denoted $\rho$.
The nonnegative inverse eigenvalue problem (or NIEP) asks for necessary and sufficient conditions
on $\sigma$ in order that $\sigma$ be the spectrum of an entry-wise nonnegative $n \times n$ matrix. 

The same problem in which we may augment the list $\sigma$ by adding an arbitrary number $N$
of zeros was solved by Boyle and Handelman \cite{B/H}. Using dynamical systems theory and ergodic theory, they proved the remarkable result that if

\begin{enumerate}[(a)]
 \item $\sigma$ has a Perron root $\lambda_1 > |\lambda_j |$ (\text{all } $j > 1$) and
\item $s_k\geq 0$ for all positive integers $k$ (and $s_m = 0$ for some $m$ implies $s_d=0$ for all positive divisors $d$ of $m)$
\end{enumerate}
then
$$\sigma_N := (\lambda_1,\ldots,\lambda_n,0,\ldots, 0) \  (N \text{ zeros added})$$
is realizable for all sufficiently large $N$.\\
In this paper we compare three constructive techniques for realizing non-real spectra with exactly three non-zero elements. Section 2 examines the companion matrix, gives a surprising result for a realizing companion matrix and provides some examples and remarks. Section 3 examines a construction due to Laffey \cite{Laf3}. Finally in section 4 we compare three constructive methods including a multi-block companion matrix in an attempt to minimize the dimension of realizing matrices with three non-zero element spectra.
\section{Known constructions}
While Boyle and Handelman's result was an existence one, in recent years more attention has been paid to trying to find a constructive realizing matrix for given spectra \cite{Laf3}. \\
This paper outlines three such constructions and compares each in turn. Our first construction involves the \emph{companion matrix}. 
\begin{defn}
If $f(x)=x^n+p_1x^{n-1}+p_2x^{n-2}+\cdots+p_n,$ then the companion matrix of $f(x)$ is
$$A(f)=\npmatrix{ 0 & 1 & 0 & \ldots & 0  \cr 0 & 0 & 1 & \ddots & \vdots \cr
\vdots & \vdots & \ddots & \ddots & 0 \cr 0 & 0 & \ldots & 0 & 1 \cr -p_n &
-p_{n-1} & \ldots & -p_2 & -p_1 \cr}.$$
\end{defn}
Note that the companion matrix has the property that det$(xI-A(f))=f(x)$.\\
Clearly, $A(f)$ is nonnegative if and only if $p_i \leq 0$ for $i=1,\ldots,n.$\\

It is known that spectra of the type $\sigma=(\rho, a\pm ib)$ where $b \neq 0$, $\rho \geq | a+ ib|$, are realizable if and only if they are realizable by a matrix of the form $\alpha I_3+C$ where $\alpha=\frac{\rho+2a}{3}\geq 0$, $I_3$ is the $3 \times 3$ identity matrix and $C$ is a nonnegative trace zero companion matrix. \\
This was first proved by Loewy and London in \cite{LL}.

\pagebreak
However this construction does not hold for $n=4$ with spectra of the form $\sigma=(\rho,a\pm ib, \mu)$ where $\mu \neq 0$, since for example the list $(4,2+i,2-i,3)$ is realizable by 
$$ \left(
  \begin{array}{ c c c c }
     \frac{8}{3} & 1 & 0 & 0   \\
      0 & \frac{8}{3} & 1 & 0  \\
	\frac{52}{27} & \frac{1}{3} & \frac{8}{3} & 0 \\
	0 & 0 & 0 & 3
  \end{array} \right)
$$
but the list $\sigma-\frac{11}{4}=(\frac{5}{4},\frac{-3}{4}+i,\frac{-3}{4}-i,\frac{1}{4})$ is not realizable by $C$ where $C$ is a nonnegative trace zero companion matrix since

\begin{eqnarray*}
f(x)&=&\Big(x-\frac{5}{4}\Big)\Big(x-\frac{1}{4}\Big)\left(\Big(x+\frac{3}{4}\Big)^2+1\right) \\
&=&x^4-\frac{3x^2}{8}-\frac{15x}{8}+\frac{125}{256}
\end{eqnarray*}
has its constant term positive and so the companion matrix of $f(x)$ is not nonnegative.\\
However for $\mu=0$ we do have the following:
\begin{Pro}
Suppose $\sigma=(\rho, a \pm ib, 0)$ is realizable, then $\sigma$ is realizable by $\alpha I_4 +C$ where $\alpha \geq 0$ and $C$ is a nonnegative trace zero companion matrix if and only if $\rho \geq 2a$.
\end{Pro}
\begin{proof}

Let 
\begin{eqnarray*}
S_k&=&\rho^k+(a+ib)^k+(a-ib)^k, \text{ and} \\
s_k&=&(\rho-\alpha)^k+(a-\alpha+ib)^k+(a-\alpha-ib)^k+(-\alpha)^k, \text{ for } k=1,2,\ldots
\end{eqnarray*}
where $\alpha=\frac{\rho+2a}{4}$. \\
Then $s_1=0$, and by the Newton identities \cite{Newton}, the characteristic polynomial of $C$ is
$$x^4-\frac{s_2}{2}x^2-\frac{s_3}{3}x+d$$
where $d=det(C)=-(\rho-\alpha)\left((a-\alpha)^2+b^2\right)\alpha \leq 0$.\\

Note $s_2=\frac{4S_2-S_1^2}{4}\geq 0$ (using $\alpha=\frac{\rho+2a}{4}$), by \eqref{JLL} for $n=4,m=2,k=1$.\\
So $C\geq 0$ if and only if $s_3\geq0$.\\
But $s_3=\frac{3}{8}(\rho-2a)(\rho^2+4b^2)$.\\
Hence $\sigma$ is realizable as $\alpha I_4+C$ if and only if $\rho \geq 2a$.
\end{proof}
\pagebreak

\subsection{Examples}
One can easily find examples of the form $(\rho, a\pm ib)$ with $\rho <2a$, which are realizable. \\
For example $(12,9\pm i)$ is realizable by 
$$ A=\left(
  \begin{array}{ c c  c }
     10 & 1 & 0 \\
     0 & 10 & 1 \\
     4 & 2 & 10	
  \end{array} \right)
$$
and so $(12,9\pm i,0)$ is realizable by the $4\times 4$ nonnegative matrix $A\oplus 0_1$.\\
Note that we can also find lists $(\rho, a \pm ib)$ such that $(\rho, a \pm ib,0)$ is not realizable but $(\rho, a \pm ib,0,0)$ is realizable. The example $\sigma=(21, 8 \pm 12i,0)$ is not realizable since it fails the JLL condition $4s_2-s_1^2\geq 0$. In fact $4s_2-s_1^2=-245$. \\
However, $\sigma_1=(21, 8 \pm 12i,0,0)$ is realizable, for example, by $\alpha I_5+C$ where $\alpha=\frac{37}{5}$ and $C$ is the trace zero companion matrix of the polynomial 
$$x^5-\frac{18x^3}{5}-\frac{9892x^2}{25}-\frac{2532243x}{125}-\frac{335969028}{3125}.$$
Similarly $\sigma=(5,2\pm 3i)$ has $4s_2-s_1^2=4(15)-9^2=-21<0$ and $5s_2-s_1^2=-6$ (while $6s_2-s_1^2>0$) but $\sigma$ with three zeros added is realizable by $\frac{3}{2}I_6+C$ where $C$ is the nonnegative $6\times 6$ trace zero companion matrix of
$$x^6-\frac{3x^4}{4}-2x^3-\frac{1197x^2}{16}-\frac{351x}{2}-\frac{6993}{64}.$$
\subsection{Remarks}
\emph{\underline{Additional remarks on non-real spectra with exactly three non-zero entries}}\\

1. Suppose that $\sigma=(\rho, a \pm ib)$, where $\rho, a, b$ are real, $i=\sqrt{-1}$, $b>0$ and $\rho \geq \sqrt{a^2+b^2}$.\\
Write $s_k=\rho^k+(a+ib)^k+(a-ib)^k$. \\
If $a \leq 0$, then $\sigma$ with $N$ zeros appended is realizable if and only if $(N+3)s_2\geq s_1^2$
where $N$ is the smallest positive integer with this property. In this case $\sigma$ is realizable by $\alpha I_{N+3}+C$ where $\alpha=\frac{\rho+2a}{N+3}$ and $C$ is a nonnegative $(N+3) \times (N+3)$ companion matrix with trace zero \cite{LafSmi}. \\

2. Suppose $a>0$. By replacing $\sigma$ by $\sigma'=(k\rho, ka \pm ikb)$ where $k=\frac{1}{\mid a^2+b^2\mid}^\frac{1}{2}$, we may assume that $a^2+b^2=1$, since $X$ is a realizing matrix for $\sigma$ if and only if $kX$ is a realizing matrix for $\sigma'$. So assume $a>0$, $\rho \geq 1$ and $\sigma=(\rho, a \pm ib)$ where $b>0$ and $a^2+b^2=1$.\\
We may write $a=\cos\theta$, $b=\sin\theta$ where $0 \leq \theta < \frac{\pi}{2}$. Then $\sigma$ with sufficiently many zeros added is the spectrum of a nonnegative matrix if and only if $\rho^k+2\cos k\theta >0$ for all positive integers $k$ by Boyle and Handelman \cite{B/H}. Because of the periodicity of the cosine function, this condition can be checked in finitely many steps. In particular, when $\theta=\frac{\pi}{l}$ for some positive integer $l$, let $\rho>\rho_0=max_{0\leq k<\lfloor\frac{l}{2}\rfloor+1}(2\cos\frac{k\pi}{l})^{\frac{1}{l-k}}$. 

\pagebreak
If $\sigma$ is realized by an $M\times M$ nonnegative matrix then the following  JLL condition must hold
$$M^{l-k-1}(\rho^{l-k}-2\cos\frac{k\pi}{l})\geq (\rho+2\cos\frac{\pi}{l})^{l-k}.$$ 
Choosing $\rho-\rho_0$ to be a sufficiently small positive number, we can realize such $\sigma$ by an $M\times M$ nonnegative matrix with the minimum dimension $M$ satisfying this JLL inequality.\\

3. Suppose that $\sigma=(\rho, a \pm ib)$, where $a>0, b>0, a^2+b^2=1$, is realizable with sufficiently many zeros added. One can ask whether $\sigma$ (with sufficiently many zeros added) is realizable by $\alpha I_N+C_N$ where $\alpha>0$ and $C_N$ is a nonnegative $N \times N$ trace zero companion matrix for some positive integer $N$. When $N=3$, it is well-known (\cite{LL}) that this realization is possible if and only if $\rho \geq a+b\sqrt{3}$ (since $s_1=0$ here). \\
When $N=4$, Proposition 2 says $\rho \geq 2a$ is necessary and sufficient.\\
We now present:

\begin{Pro}
Suppose that $\sigma=(\rho, a \pm ib)$, $a>0,b>0, a^2+b^2=1$, is realizable with sufficiently many zeros added. If $\sigma$ (with sufficiently many zeros added) is realizable by $\alpha I_N+C_N$ where $C_N$ is a nonnegative $N \times N$ trace zero companion matrix for all sufficiently large $N$, then $\rho \geq 2a$ (so $\sigma$ with only one zero added is realizable). 
\end{Pro}
\begin{proof}

Let $\sigma_N=(\rho,a+ib,a-ib,\underbrace{0,\ldots,0}_{\rm N-3 \ zeros} \}$
.\\
Let $\alpha=\frac{\rho+2a}{N}$, and $S_k=tr (C_N)^k$, so
\begin{eqnarray*}
S_k&=&(\rho-\alpha)^k+(a-\alpha+ib)^k+(a-\alpha-ib)^k+(-\alpha)^k+(-\alpha)^k+\cdots+(-\alpha)^k.
\end{eqnarray*}
In order that $\sigma$ with $N-3$ zeros added be realizable in the desired form for all sufficiently large N, we require $C_N$ to be nonnegative i.e. the coefficients $P_j$ must be non-positive for \\
$j=1,2,\ldots N$, in det$(xI-C_N)$
\begin{eqnarray*}
&=&x^N+P_1x^{N-1}+P_2x^{N-2}+\cdots+P_N \\
&=& (x-\tau)(x^2-\frac{2\left((N-2)a -\rho\right)}{N}x+1)(x+\frac{\rho+2a}{N})^{N-3}
\end{eqnarray*}
where $\tau=\frac{(N-1)\rho-2a}{N}$. \\
The coefficient of $x^{N-1}$ is zero since $C_N$ has trace 0.\\
Consider the coefficient of $x^{N-4}$.\\
The second Newton identity \cite{Newton} states that $S_2+S_1P_1+2P_2=0$, but $S_1=0$ and so we have that $P_2=-\frac{S_2}{2}$.
\\ 
The fourth Newton identity states that $S_4+S_3P_1+S_2P_2+S_1P_3+4P_4=0$ which simplifies to $S_4-\frac{S_2^2}{2}+4P_4=0$.\\
Hence the coefficient $P_4$ of $x^{N-4}$ in det$(xI-C_N)$ is $-\frac{1}{4}\left(S_4-\frac{S_2^2}{2}\right)$ and for realizability by a companion matrix, for sufficiently large $N$, this coefficient must not be positive i.e. $\left(S_4-\frac{S_2^2}{2}\right)$ must be nonnegative. 

\pagebreak
Expanding $\left(S_4-\frac{S_2^2}{2}\right)$ we get $\frac{(2a+\rho)(N-3)}{2N^3}(f(N))$ where $f(N)$ is a quadratic polynomial in $N$ in terms of $\rho, a$ and $b$ which equals \\
$$(\rho-2a)(\rho^2+4b^2)N^2+((4\rho a^2+8\rho b^2 +16ab^2-(3\rho^3+2\rho^2a+8a^3))N+12\rho^2a+24\rho a^2+2\rho^3+16a^3.$$
The leading coefficient $(\rho-2a)(\rho^2+4b^2)$ is nonnegative when $\rho \geq 2a$. \\
But from Proposition 2 this is exactly the condition for $\sigma \cup \{0\}$ to be realizable in the desired form and this establishes the result.


\end{proof}
\begin{Rem*}
The negativity of the coefficient of $x^{N-4}$ in Proposition 3 turns out to be the most restrictive condition in establishing the result. This proposition demonstrates that when adding lots of zeros to $\sigma$ to aid realizability, the method $\alpha I_N + C_N$ for nonnegative companion matrices $C_N$ is not the best strategy. \\ 
The realizability by matrices of the form (see the next section)
$$X_N=\npmatrix{ x_1 & 1 & 0 & \ldots & & 0  \cr x_2 & x_1 & 2 & 0 &\ldots &   0 \cr
x_3 & x_2 & x_1 & 3 & 0  & 0 \cr \vdots & \ddots & \ddots & \ddots & \ddots \cr x_{N-1} &
x_{N-2} & \ldots & x_2 & x_1 &  N-1 \cr x_N & x_{N-1} & \ldots & x_3 & x_2 &  x_1}$$
is possible for all sufficiently large $N$, but leads to very high dimensional realizations. 
\end{Rem*}

\section{Laffey's construction}
In \cite{Laf3} Laffey gives a constructive method for the Boyle-Handelman theorem. He shows that if  $\tau=(\mu_1,\ldots, \mu_n),$ 
\begin{eqnarray*}
x_k:&=&\mu_1^k+\cdots+\mu_n^k, \ k = 1, 2, 3,\ldots,\\
 q(x):&=&(x-\mu_1)\cdots(x-\mu_n)\\
&=&x^n+q_1x^{n-1}+q_2x^{n-2}+\cdots+q_{n-1}x+q_n
\end{eqnarray*}
then the matrix
$$X_n=\npmatrix{ x_1 & 1 & 0 & \ldots & & 0  \cr x_2 & x_1 & 2 & 0 &\ldots &   0 \cr
x_3 & x_2 & x_1 & 3 & 0  & 0 \cr \vdots & \ddots & \ddots & \ddots & \ddots & \vdots \cr x_{n-1} &
x_{n-2} & \ldots & x_2 & x_1 &  n-1 \cr x_n & x_{n-1} & \ldots & x_3 & x_2 &  x_1}$$
has characteristic polynomial
$$ Q (x)=x^n +nq_1x^{n-1}+n(n-1)q_2x^{n-2} +\cdots+n!q_n.$$
Thus it follows that the spectrum of $Q(x)$ is realizable if the $x_i, \ (i=1,2,\ldots,n),$ are nonnegative.
Hence if we wish to realize a given spectrum $\sigma=(\lambda_1,\lambda_2,\ldots,\lambda_n)$, let
$$q(x)=x^n+q_1x^{n-1}+q_2x^{n-2}+\cdots+q_{n-1}x+q_n$$ with
$$q_i=\frac{p_i}{n(n-1)\cdots(n-i+1)}  \text{ for } i=1,2,\ldots,n$$
then $q(x)$ has its corresponding 
\begin{eqnarray*}
Q(x)=f(x)&=&x^n+p_1x^{n-1}+\cdots+p_{n-1}x+p_n \\
&=&(x-\lambda_1)\cdots(x-\lambda_n).
\end{eqnarray*}
Laffey gives a lower bound for $N$, the number of zeros required to be added for realizability, in the proof of his main result \cite{Laf3} depending on a number of parameters.\\
This bound is not optimal in general however. \\
For example the list $\sigma=(1.1, e^{\pm i \theta})$ for $\theta=0.0188\pi$ cannot be realized by a $3\times3$ matrix by Remark 3 of Section 2.2 since $\rho \ngeq a+b\sqrt{3}$. \\This list with one zero added can however be realized by the nonnegative $4 \times 4$ matrix  
$$ \left(
  \begin{array}{ c c c c }
     1.013005334 & 1 & 0 & 0   \\
      0 & 1.041605274 & 1 & 0  \\
	0 & 0 & 1.041605274 & 1 \\
	0.000326227 & 0 & 0 & 0.000296825
  \end{array} \right).
$$
This construction arises from the $4 \times 4$ solution of the NIEP via examination of the coeffiecients of the characteristic polynomial of realizing matrices given in \cite{Spain}.

Thus $\sigma \cup \{0\}$  is realizable whereas the minimum number of zeros $N$ required for realizability in Laffey's construction is $N=198$. The condition that all the main diagonal entries of $X_n$ are equal is quite restrictive and allowing these entries to be any nonnegative entries that add up to the sum of the elements in $\sigma$ is less restrictive in this sense. Here cos$(\theta)=99825635046\cdots$ and note that realization becomes more difficult when the spectral gap i.e. the distance between the Perron root and the second biggest eigenvalue - decreases.
\section{Three constructive methods}

Suppose that $\sigma=(\rho, \lambda_2, \overline{\lambda_2})$, where $\rho >|\lambda_2|$ and $\lambda_2,\overline{\lambda_2}$ are non-real complex conjugates, and that all the power sums $s_k>0$. By the Boyle-Handelman theorem \cite{B/H}, there exists a positive integer $N$ such that $\sigma \cup N$ zeros is realizable. In this section, we consider finding such an $N$ and a corresponding realizing matrix for $\sigma \cup N$ zeros. \\
\\
\emph{\underline{Method 1: $\alpha I+C$}}\\
As we saw in the previous sections one method is to try to realize  $\sigma \cup N$ zeros by $\alpha I_m +C,$ where $m=N+3, \alpha \geq 0$, and $C$ is a nonnegative trace zero companion matrix. \\
For instance, in the Examples section we saw that for $\sigma=(5,2+3i,2-3i)$, $\sigma\cup 3$ zeros is realizable in this way and that $(5,2+3i,2-3i,0,0)$ is not realizable. So this method yields the best possible result in this case. \\
However, this method need not always work.\\
For example, let $\sigma=(1.4,e^{\pm i\theta})=(\rho,a+ib, a-ib)$ where $\cos \theta=0.95$ (note that $\rho \ngeq 2a)$. Then 
\begin{align*}
 f(x)=(x-1.4)(x^2-1.9x+1)\\=x^3-3.3x^2+3.66x-1.4\\ =x^3-(\rho+2a)x^2+(1+2a\rho)x-\rho
\end{align*}

If $x^Nf(x)$ is to be realizable by $\alpha I_m+C$, where $m=N+3$, then for $x=y+\frac{2a+\rho}{N+3}$, the polynomial 
\begin{align*}
g(y)=\left(y+\frac{2a+\rho}{N+3}\right)^{N+3}-(\rho+2a)\left(y+\frac{2a+\rho}{N+3}\right)^{N+2}\\+(1+2a\rho)\left(y+\frac{2a+\rho}{N+3}\right)^{N+1}-\rho\left(y+\frac{2a+\rho}{N+3}\right)^N 
\end{align*}

must have all its coefficients non-positive. \\
Using the substitution $y\rightarrow \frac{1}{t}$ and multiplying $g(y)$ by $t^{N+3}$ we get
\begin{align*}
\left(1+\frac{2a+\rho}{N+3}t\right)^{N+3}-(\rho+2a)t\left(1+\frac{2a+\rho}{N+3}t\right)^{N+2}\\+(1+2a\rho)t^2\left(1+\frac{2a+\rho}{N+3}t\right)^{N+1}-\rho t^3\left(1+\frac{2a+\rho}{N+3}t\right)^N. 
\end{align*}
We note that the companion matrix of $F(x)=x^Nf(x)$ is nonnegative if and only if $f(t)=t^{N+3}F(\frac{1}{t})$ has all its coefficients non-positive. Expanding this expression we see that at least one of the coefficients of $t$ is positive.\\
Note that if $\cos\theta<0$ then $\sigma \cup N$ zeros is realizable for the smallest nonnegative integer $N$ satisfying $(N+3)s_2\geq s_1^2$, see \cite{LafSmi}.\\
\\
\emph{\underline{Method 2: $X_m$}}\\
Laffey \cite{Laf3} showed that when $\lambda_1=\rho>|\lambda_j|, j>1$ for $\sigma=(\rho, \lambda_2,\ldots,\lambda_n)$ and all $s_k>0,$ there exists $N\geq 0$ such that $\sigma \cup N$ zeros is realizable by an $m \times m$ matrix of the form
$$X_m=\npmatrix{ x_1 & 1 & 0 & \ldots & & 0  \cr x_2 & x_1 & 2 & 0 &\ldots &   0 \cr
x_3 & x_2 & x_1 & 3 & 0  & 0 \cr \vdots & \ddots & \ddots & \ddots & \ddots \cr x_{m-1} &
x_{m-2} & \ldots & x_2 & x_1 &  m-1 \cr x_m & x_{m-1} & \ldots & x_3 & x_2 &  x_1}$$
with $m=N+3$.\\
The $x_j$ are recovered inductively from the equations $trX_m^k=s_k, k=1,2,3,\ldots,m.$\\
For example, for $\sigma=(1.4,e^{\pm i\theta})$ where $\cos \theta=0.95,$ this method succeeds with $N=9$.\\
The minimum $N$ for which $\sigma\cup N$ zeros is realizable is not known in this case; $N=9$ is the minimum for which realizability by a matrix of the form $X_m$ is possible. While this method is systematic and easy to carry out, it is not guaranteed to yield the minimum $N$ for which $\sigma\cup N$ zeros is realizable however.\\
\\
\emph{\underline{Method 3: Multi-block companion matrices}}\\
Another method which may be attempted, is to realize $\sigma \cup N$ zeros by a matrix of the form\\

$$M=\LARGE{ \left(
\begin{array}{c | c | c | c | c }
C_1 & N_1 & 0   & 0  & 0  \\ \hline
0   & C_2 & N_2 & 0 &   0 \\ \hline
\vdots & \ddots & \ddots & \ddots & \vdots \\ \hline
0 & 0 & 0 & \ddots & N_{r-1}   \\ \hline 
R_1 & R_2 & \cdots & R_{r-1} & C_r    
\end{array}\right)}$$
\\

where 
\begin{itemize}
\item 
the matrix blocks $C_i$ are nonnegative companion matrices of size $k_i \times k_i$ for $i=1,2,\ldots,r\geq 1$
\item the matrices $N_i$ are size $k_i \times k_{i+1}$ (for $i=1,2,\ldots,r-1$), which have their $(k_i,1)$ entry equal to one and all other entries equal to zero and 
\item the matrix $R_i$ is the $k_r \times k_i$ matrix whose rows are all equal to zero, except the last row which is equal to $(r_{i1} \ r_{i2} \ \cdots \ r_{ik_i})$. 
\end{itemize}
Such matrices were first studied by Laffey \cite{Laf2}, and by Laffey and Smigoc \cite{LafSmi4}, in work on spectra with at most two positive elements and all other elements having negative real parts. They can also be employed in obtaining realizations over the real numbers of the spectra realized by Kim, Ormes and Roush \cite{Kim}.
Unfortunately there is no known ``good'' algorithm for choosing $r$ or the companion matrices $C_i$. \\
A result (Theorem 5) of Laffey, Loewy and Smigoc in \cite{LafLoeSmi}, implies that if $\sigma=(\lambda_1,\lambda_2,\overline{\lambda_2})$ satisfies the hypotheses given earlier in this section, i.e.\
\begin{enumerate}
\item $s_k(\sigma)>0$
\item  $\lambda_1>|\lambda_2|$
\end{enumerate}
then if
$$\widetilde{f}(t)=(1-\lambda_1t)(1-\lambda_2t)(1-\overline{\lambda_2}t)$$ we have that $1-\widetilde{f}(t)^{\frac{1}{N}}$ has positive coefficients for sufficiently large integers $N$. \\
The $N$ arising in the proof is the minimum $N$ for which $\sigma \cup (N-3)$ zeros is realizable by a matrix of the form $X_m$ above.\\
Writing $\widetilde{f}(t)^{\frac{1}{N}}=1-l_1t-l_2t^2- \cdots - l_Nt^N- \cdots$ where $l_i \geq 0,$ one choice is to take $r=N$, $C_i$ the companion matrix of 
$$e(x):=x^N-l_1x^{N-1}-l_2x^{N-2}-\cdots - l_N \text{ for } i=1,2,\ldots ,N$$
and then determine the entries in the last row of $M$ in order that $det(xI-M)=x^{N^2-3}f(x)$ where $f(x)=(x-\lambda_1)(x-\lambda_2)(x-\overline{\lambda_2}).$\\
This method succeeds precisely when the entries in position 
$$(N^2,1),(N^2,2),\ldots,(N^2,N^2-N)$$ turn out to be nonnegative. This can be done algorithmically as follows:\\

\emph{\underline{The Algorithm}}\\
Calculate the remainder $r_1$ of $x^{N^2-3}f(x)$ upon division by $e(x)$, then the quotient $q_1$ of $x^{N^2-3}f(x)$ upon division by $e(x)$, then the remainder $r_2$ of $q_1$ upon division by $e(x)$, the quotient $q_2$ of $q_1$ upon division by $e(x)$, and in general, for $i>1$, $r_i$ the remainder $r_i$ of $q_{i-1}$ upon division by $e(x)$ and the quotient $q_i$ of $q_{i-1}$ upon division by $e(x)$ concluding with $r_{N-1}$ the remainder of $q_{N-2}$ upon division by $e(x)$. \\
The entries in position $(N^2,j), j=1,2,\ldots,N^2-N$ are the minus of the coefficients of the powers of $x^i$ $(i=0,1,\ldots,N-1)$ occurring in $r_u, u=1,2,\ldots, N-1$.\\

For the example given earlier of $f(x)=x^3-3.3x^2+3.66x-1.4$ corresponding to $\sigma=(1.4,e^{\pm i\theta})$ where $\cos \theta=0.95$, we find that $1-\widetilde{f}(t)^{\frac{1}{4}}$ has the first four coefficients positive (in fact it can be shown that all the coefficients of the Taylor series expansion of $1-\widetilde{f}(t)^{\frac{1}{4}}$ are positive) and the corresponding $16 \times 16$ matrix $M$ is found to have nonnegative entries. So one obtains a realization of $\sigma$ with $13$ zeros added in this case. \\
While in this example, the realization of $\sigma$ by a matrix of type $X_N$ requires fewer (nine) zeros, it is unknown which method provides the best bound for other given $\sigma$.\\

\pagebreak
For methods 2 and 3 mentioned above, we have that for $\cos(\theta)=0.95$\\
$\sigma_9=(1.4,e^{\pm i\theta},0,0,0,0,0,0,0,0,0)$ is realizable by
$$ X_{12}=\left(
  \begin{array}{ c c c c c c c c c c c c }
     x_1 & 1 & 0 & 0 & 0 & 0 & 0 & 0 & 0 & 0 & 0 & 0     \\
      x_2 & x_1 & 2 & 0 & 0 & 0 & 0 & 0 & 0 & 0 & 0 & 0    \\
	x_3 & x_2 & x_1 & 3 & 0 & 0 & 0 & 0 & 0 & 0 & 0 & 0   \\
	x_4 & x_3 & x_2 & x_1 & 4 & 0 & 0 & 0 & 0 & 0 & 0 & 0  \\
      x_5 & x_4 & x_3 & x_2 & x_1 & 5 & 0 & 0 & 0 & 0 & 0 & 0  \\
      x_6 & x_5 & x_4 & x_3 & x_2 & x_1 & 6 & 0 & 0 & 0 & 0 & 0  \\
      x_7 & x_6 & x_5 & x_4 & x_3 & x_2 & x_1 & 7 & 0 & 0 & 0 & 0  \\
      x_8 & x_7 & x_6 & x_5 & x_4 & x_3 & x_2 & x_1 & 8 & 0 & 0 & 0  \\ 
      x_9 & x_8 & x_7 & x_6 & x_5 & x_4 & x_3 & x_2 & x_1 & 9 & 0 & 0  \\
      x_{10} & x_9 & x_8 & x_7 & x_6 & x_5 & x_4 & x_3 & x_2 & x_1 & 10 & 0  \\
      x_{11} & x_{10} & x_9 & x_8 & x_7 & x_6 & x_5 & x_4 & x_3 & x_2 & x_1 & 11  \\
      x_{12} & x_{11} & x_{10} & x_9 & x_8 & x_7 & x_6 & x_5 & x_4 & x_3 & x_2 & x_1  \\
\end{array}
 \right)$$
where 
\begin{eqnarray*}
x_1&=&\frac{11}{40}, \ x_2=\frac{71}{3,520}, \ x_3=\frac{777}{704,000}, \ x_4=\frac{33,371}{929,280,000}, \\ 
x_5&=&\frac{8,251}{12,390,400,000}, \ x_6=\frac{1,171,069}{3,271,065,600,000}, \ x_7=\frac{231,739,033}{1,962,639,360,000,000}, \\ 
x_8&=&\frac{20,078,111,833}{863,561,318,400,000,000}, \ x_9=\frac{120,886,554,859}{34,542,452,736,000,000,000}, \\ x_{10}&=&\frac{807,900,538,537}{1,823,841,504,460,800,000,000}, \ x_{11}=\frac{8,197,245,662,72}{165,803,773,132,800,000,000,000}, \\ 
x_{12}&=&\frac{1,344,264,039,555,553}{267,496,753,987,584,000,000,000,000} 
\end{eqnarray*}
and $\sigma_{13}=(1.4,e^{\pm i\theta},0,0,0,0,0,0,0,0,0,0,0,0,0)$ is realizable by
$$ M_{16}=\left(
  \begin{array}{ c c c c c c c c c c c c c c c c }
     0 & 1 & 0 & 0 & 0 & 0 & 0 & 0 & 0 & 0 & 0 & 0 & 0 & 0 & 0 & 0    \\
      0 & 0 & 1 & 0 & 0 & 0 & 0 & 0 & 0 & 0 & 0 & 0 & 0 & 0 & 0 & 0   \\
	0 & 0 & 0 & 1 & 0 & 0 & 0 & 0 & 0 & 0 & 0 & 0 & 0 & 0 & 0 & 0  \\
	a & b & c & d & 1 & 0 & 0 & 0 & 0 & 0 & 0 & 0 & 0 & 0 & 0 & 0 \\
0 & 0 & 0 & 0 & 0 & 1 & 0 & 0 & 0 & 0 & 0 & 0 & 0 & 0 & 0 & 0 \\
0 & 0 & 0 & 0 & 0 & 0 & 1 & 0 & 0 & 0 & 0 & 0 & 0 & 0 & 0 & 0 \\
0 & 0 & 0 & 0 & 0 & 0 & 0 & 1 & 0 & 0 & 0 & 0 & 0 & 0 & 0 & 0 \\
0 & 0 & 0 & 0 & a & b & c & d & 1 & 0 & 0 & 0 & 0 & 0 & 0 & 0 \\
0 & 0 & 0 & 0 & 0 & 0 & 0 & 0 & 0 & 1 & 0 & 0 & 0 & 0 & 0 & 0 \\
0 & 0 & 0 & 0 & 0 & 0 & 0 & 0 & 0 & 0 & 1 & 0 & 0 & 0 & 0 & 0 \\
0 & 0 & 0 & 0 & 0 & 0 & 0 & 0 & 0 & 0 & 0 & 1 & 0 & 0 & 0 & 0 \\
0 & 0 & 0 & 0 & 0 & 0 & 0 & 0 & a & b & c & d & 1 & 0 & 0 & 0 \\
0 & 0 & 0 & 0 & 0 & 0 & 0 & 0 & 0 & 0 & 0 & 0 & 0 & 1 & 0 & 0 \\
0 & 0 & 0 & 0 & 0 & 0 & 0 & 0 & 0 & 0 & 0 & 0 & 0 & 0 & 1 & 0 \\
0 & 0 & 0 & 0 & 0 & 0 & 0 & 0 & 0 & 0 & 0 & 0 & 0 & 0 & 0 & 1 \\
r_{11} & r_{12} & r_{13} & r_{14} & r_{21} & r_{22} & r_{23} & r_{24} & r_{31} & r_{32} & r_{33} & r_{34} & a & b & c & d \\
  \end{array} \right)
$$
where 
\begin{eqnarray*}
a&=&\frac{171,081}{4,096,000}, \ b=\frac{6,487}{128,000}, \ c=\frac{339}{3,200}, \ d=\frac{33}{40}, \\
r_{11}&=&\frac{51,022,761,666,057,454,968,319,563}{35,184,372,088,832,000,000,000,000,000}, \ r_{12}=\frac{874,779,856,035,649,816,558,659}{274,877,906,944,000,000,000,000,000}, \\
r_{13}&=&\frac{46,738,794,894,222,413,568,447}{6,871,947,673,600,000,000,000,000}, \ r_{14}=\frac{3,037,818,107,230,325,277,507}{85,899,345,920,000,000,000,000}, \\
r_{21}&=&\frac{43,164,064,793,619,336,981}{1,073,741,824,000,000,000,000}, \ r_{22}=\frac{9,443,904,869,138,337}{209,715,200,000,000,000
}, \\ 
r_{23}&=&\frac{9,378,192,467,558,799}{167,772,160,000,000,000}, \ r_{24}=\frac{212,079,653,123}{1,310,720,000,000}, \\
r_{31}&=&\frac{1,115,284,117,329}{8,388,608,000,000}, \ r_{32}=\frac{14,100,563,727}{131,072,000,000}, \\ 
r_{33}&=&\frac{1,488,555,033}{16,384,000,000}, \ r_{34}=\frac{39,083,751}{204,800,000}.
\end{eqnarray*}

\section{Conclusion}
In this paper we give a necessary and sufficient condition for a non-real list with three non-zero numbers to be realized by a companion matrix. We also offered a comparison of three constructive methods in realizing non-real spectra with exactly three nonzero elements.
While the least $N$ in \cite{Laf3} is still unknown, the methods here may help in finding such $N$, particularly in the case where $\sigma$ contains just three nonzero elements. The fact that the case of just \textit{three} non-zero element spectra does not yield much more progress than in the general NIEP indicates just how intriguing this problem remains. The problems discussed in this paper while very special, can serve as a fertile testing ground for new constructive techniques in the NIEP. 

\section{Acknowledgements} The author is very grateful for the hospitality he received at the Technion (under the Erasmus Mundus Action 2 EMAIL III mobility programme) during the completion of this paper and would like to sincerely thank Professor Tom Laffey and Professor Raphi Loewy for helpful conversations on this paper and ideas to improve its readability. Thanks also to the anonymous referee who gave this paper a very considered treatment and provided helpful comments.
 
\pagebreak

\end{document}